\documentclass[a4paper,11pt]{article}
\usepackage{amsmath,amssymb,amsthm,graphicx}
\usepackage{fullpage}
\usepackage{enumitem, verbatim}
\usepackage{xcolor}

\newcommand{\F}{\mathbb{F}}
\newcommand{\Z}{\mathbb{Z}}
\newcommand{\OOC}{\mathrm{OOC}}
\newcommand{\CW}{\mathrm{CW{\text -}CPC}}
\newcommand{\code}{{\mathcal{C}}}
\newcommand{\bie}{{\mathcal{B}}}
\newcommand{\design}{{\mathcal{D}}}

\newtheorem{theorem}{Theorem}[section]
\newtheorem{lemma}[theorem]{Lemma}
\newtheorem{corollary}[theorem]{Corollary}

\theoremstyle{definition}
\newtheorem{example}[theorem]{Example}
\newtheorem{definition}[theorem]{Definition}
\newtheorem{remark}[theorem]{Remark}

\newtheorem{proposition}[theorem]{Proposition}
\newtheorem{note}[theorem]{Note}

\title{Optical orthogonal codes from a combinatorial perspective}
\author{Sophie Huczynska\footnote{School of Mathematics and Statistics, University of St Andrews, St Andrews, Fife, KY16 9SS,UK, sh70\@st-andrews.ac.uk} \/ and \/  Siaw-Lynn Ng \footnote{Information Security Group, Royal Holloway, University of London, Egham, Surrey, TW20 0EX, UK, S.Ng\@rhul.ac.uk}}
\begin{document}
\maketitle

\begin{abstract}
Optical orthogonal codes (OOCs) are sets of $(0,1)$-sequences with good auto- and cross- correlation properties.  They were 
originally introduced for use in multi-access communication, particularly in the setting of optical CDMA communications systems.  They can also be formulated in terms of families of subsets of $\Z_v$, where the correlation properties can be expressed in terms of conditions on the internal and external differences within and between the subsets.  With this link there have been many studies on their combinatorial properties.  
However, in most of these studies it is assumed that the auto- and cross-correlation values are equal; in particular, many constructions focus on the case where both correlation values are $1$.  This is not a requirement of the original communications application. 
In this paper, we ``decouple" the two correlation values and consider the situation with correlation values greater than $1$. We consider the bounds on each of the correlation values, and the structural implications of meeting these separately, as well as associated links with other combinatorial objects.  We survey definitions, properties and constructions, establish some new connections and concepts, and discuss open questions.

\end{abstract}


\section{Introduction}
\label{sec:introduction}

Optical orthogonal codes (OOCs) were introduced for use in multi-access communication, particularly in the setting of optical CDMA communications systems  (\cite{BriWei,ChuSalWei}).  Since then, they have received considerable attention, with extensions to two and three dimensions (\cite{FengWangWangZhao, FengChang, Alderson}). 

In its original form, an OOC comprises a family of $(0,1)$-sequences with good auto- and cross- correlation properties.  The sequence definition can be reformulated in terms of families of subsets of $\Z_v$, where the $0$ and $1$ entries in the sequence correspond to the absence or presence of an element in a set.  The correlation properties can then be expressed in terms of bounds on the differences within and between the subsets in the family.

This set-theoretic formulation brings OOCs within the framework of various combinatorial constructions and approaches, with a particular relationship to generalisations of difference families (\cite{NgPat}).  However, various aspects of the motivating problem ensure that the families of sets corresponding to OOCs have a very different flavour to most commonly-studied combinatorial situations involving internal and external differences.  Unlike external difference families (EDFs), OOCs have no requirement that the sets are disjoint.  Moreover, unlike classic difference families (DFs) and EDFs, the bounds refer to the number of internal differences within any single set or external differences between any pair of distinct sets, rather than to the multiset union of all internal or external differences.

In many studies it is assumed that the auto- and cross- correlation values are equal.  In particular, many constructions focus on the case where both correlation values are 1.  Here we aim to “decouple” the two correlation values and study the relationship between them.  This is mathematically interesting, but also has justification on practical grounds: in \cite{ChuSalWei, YangFuja} it is explained that the auto-correlation constraints for OOCs relate to the issue of synchronisation, while cross-correlation constrains chiefly relate to operation. Moreover, OOCs have strong connections to other combinatorial objects for which only one of the correlation values is important (for example, to constant weight cyclically permutable codes, where only the cross-correlation value is taken into account when considering the minimum distance of the code). 

In Section \ref{sec:definitions} we introduce definitions and fundamental properties of OOCs as sequences and subsets.  We clarify boundary cases, and consider the question of isomorphisms and equivalences of OOCs.
In Section \ref{sec:boundsandconstructions} we briefly survey some bounds and constructions in the literature that focus on OOCs with different auto- and cross-correlation values.  Following this, in Section \ref{sec:others} examine the relationship of OOCs with other combinatorial objects, again focusing on the roles of the auto- and cross-correlation values in the parameters of these objects.  In Section \ref{sec:bounds}
we treat the bounds on auto- and cross-correlation values separately and consider the consequences of these bounds. 
Finally, in Section \ref{sec:conclusion} we indicate some open questions. 

\section{Definitions and fundamental properties}
\label{sec:definitions}

Here we introduce the definitions of optical orthogonal codes (OOCs) as $(0,1)$-sequences and as subsets of $\Z_v$, the integers modulo $v$,  and examine some fundamental properties.

\subsection{OOCs as sequences} \label{sub:OOCsequences}

\begin{definition}[OOCs as sequences]
\label{defn:OOCsequences}
Let $v, w, \lambda_a, \lambda_c$ be non-negative integers with $v \ge 2$, $w \ge 1$.  A $(v,w,\lambda_a, \lambda_c)$-OOC $\code$ of size $N \ge 1$ is a family $\{
X_0, \ldots, X_{N-1}\}$ of $N$ $(0,1)$-sequences of length $v$, weight
$w$, such that auto-correlation values are at most $\lambda_a$ and
cross-correlation values are at most $\lambda_c$.

More formally, writing $X=(x_i)_{i=0}^{v-1}$, $Y =
(y_i)_{i=0}^{v-1}$, $x_i, y_i \in \{0, 1\}$, we require (with indices
written modulo $v$):
\begin{enumerate}
\item[(A)] $\displaystyle{\sum_{t=0}^{v-1}} x_t x_{t+\delta} \le \lambda_a \mbox{ for any } X \in \code, 0 < \delta \le v-1$, 
\item[(B)] $\displaystyle{\sum_{t=0}^{v-1}} x_t y_{t+\delta} \le \lambda_c \mbox{ for any } X, Y \in \code, 0 \le \delta \le v-1$.
\end{enumerate}
\end{definition}

If $X=(x_i)_{i=0}^{v-1}$, we will say $X+s = (x_{i+s})_{i=0}^{v-1}$ is a (cyclic) shift of $X$ by $s$ places.  We may write $X=(x_i)$, omitting the range if there is no ambiguity.
We examine first the boundary cases of $\lambda_a=0$, $\lambda_c=0$, $N=1$, and $w=1$. 
\begin{remark}[The case $N=1$]
\label{rem:N=1}
Clearly there is no cross-correlation if there is only one sequence.  In this case we will set $\lambda_c=0$. If there is more than one sequence then $\lambda_c$
cannot be 0.
$\qedsymbol$
\end{remark}

For the case $\lambda_a = 0$, it follows immediately from the definition that:

\begin{lemma}
\label{lem:lambda_a}
A $(v,w,\lambda_a, \lambda_c)$-OOC has $\lambda_a=0$ if and only if $w=1$, that is, all the sequences have exactly one 1 and $v-1$ 0s. 
\end{lemma}

The relationship between $w=1$ and $\lambda_c$ is only slightly more complicated.

\begin{lemma}
\label{lem:lambda_c}
For a $(v,1,\lambda_a, \lambda_c)$-OOC of size $N$,
\[
\lambda_c= \left\{ \begin{array}{ll} 0 & \mbox{ if } N=1, \\
                             1 & \mbox{ if } N>1.
                             \end{array} \right.
\]
\end{lemma}

The converse is not true: for example, in the case $N>1$ the OOC consisting of two sequences 1100 and 1010 has $\lambda_c=1$ but $w=2$, and in the case $N=1$, the OOC has one sequence which could have any number of 1s and hence any $w \le v$.

\begin{definition}[Non-trivial OOCs]
\label{defn:nontrivial}
We say a $(v, w, \lambda_a, \lambda_c)$-OOC $\code$ is \emph{trivial} if $\lambda_a = 0$ or $\lambda_c=0$.  We say $\code$ is \emph{non-trivial} if $\lambda_a \ge 1$ and $\lambda_c \ge 1$.
\end{definition}

It follows that a non-trivial OOC has $w > 1$ and $N >1$.

\begin{remark}[On correlation]
As an aside we note that ``correlation'' sometimes means different things in
different contexts.  In some types of sequences, such as frequency hopping
sequences, ``correlation'' means Hamming correlation, where we get a contribution of 
1 if and only if both sequences have the {\em same symbols} at the
same place.  In OOCs, we get a contribution of 1 to the correlation value if and only if both sequences have a 1 at the same place.   We will say that the sequence $X$ \emph{collides with} the sequence 
$Y$ in position $h$ if both $X$ and $Y$ has a 1 in that position, or that $X$ and $Y$ \emph{have a collision} at that position.
$\qedsymbol$
\end{remark}

By Definition \ref{defn:OOCsequences}, a $(v,w,\lambda_a, \lambda_c)$-OOC is also a
$(v,w,\lambda_a', \lambda_c')$-OOC for $\lambda_a' \ge \lambda_a$,
$\lambda_c' \ge \lambda_c$.  We shall introduce the following terminology for when
$\lambda_a$ or $\lambda_c$ is attained:

\begin{definition}[Proper OOCs] \label{defn:proper}
We say that a $(v,w,\lambda_a, \lambda_c)$-OOC is an \emph{a-proper}
 $(v,w,\lambda_a, \lambda_c)$-OOC if $\lambda_a$ is attained by some
sequence in the OOC.  We say that it is a \emph{c-proper}
 $(v,w,\lambda_a, \lambda_c)$-OOC if $\lambda_c$ is attained by some
pair of sequences in the OOC.  (We allow an OOC with $N=1$ to be c-proper.)  We say that an OOC is \emph{proper}
if it is both a-proper and c-proper.
\end{definition}

Next we define ``full cyclic order'': 

\begin{definition}[Full cyclic order] \label{defn:fullcyclic}
We say that a $(0,1)$-sequence of length $v$, $X=(x_i)_{i=0}^{v-1}$, has
full cyclic order if $(x_i)_{i=0}^{v-1} = (x_{i+\delta})_{i=0}^{v-1}$ 
happens only if $v | \delta$.
\end{definition}

\begin{definition}[Non-degenerate OOCs] \label{defn:nondegenerate}
We say that a  $(v,w,\lambda_a, \lambda_c)$-OOC is \emph{non-degenerate}
if all its sequences have full cyclic order and no sequence
is a shift of another.
\end{definition}

It is not hard to see that a single-sequence OOC can be non-degenerate, and that if $w=1$ then a non-degenerate OOC can have only one sequence.

Non-degeneracy is a concept that comes into play particularly when we consider 
equivalence - when is an OOC ``the same as'' another OOC?  We will discuss
this further in Section \ref{sub:equivalence}.  Non-degeneracy also implies 
the following:

\begin{theorem} \label{thm:<w}
For a non-degenerate proper $(v, w, \lambda_a, \lambda_c)$-OOC, 
$\lambda_a < w$, $\lambda_c < w$.
\end{theorem}
\begin{proof}
By definition $\lambda_a \le w$.  Suppose $\lambda_a = w$.  Then there is a sequence $X=(x_i)$ in the OOC with 
$\Sigma x_t x_{t+\delta} = w$ for some $0 < \delta \le v-1$.  Since
there are only $w$ 1s in $(x_i)$ that means all the 1s of $(x_i)$ and
$(x_{i+\delta})$ are in the same positions, so $(x_i)= (x_{i+\delta})$ with $v
\nmid \delta$, contradicting the assumption that $(x_i)$ has full cyclic
order. The $\lambda_c$ case is similar.
\end{proof}

\subsection{OOCs as subsets} \label{sub:OOCsubsets}

OOCs may be written in terms of subsets of $\Z_v$.  This correspondence occurs frequently in the literature.

Let $X=(x_i)$ be a binary sequence of length $v$.  Let $Q$ be the set of integers modulo $v$ denoting the positions of the non-zero bits of $X$. We write $Q+\delta = \{q+\delta \pmod{v} : q \in Q\}$.  Consider $x_t x_{t+\delta}$, $\delta \in \Z_v \setminus \{0\}$. We have $x_t x_{t+\delta}=1$ if and only if $x_t = x_{t+\delta} =1$, that is, $t \in Q$, $t+\delta \in Q$, and so $t \in Q \cap (Q-\delta)$.  This allows us to write OOCs as subsets of $\Z_v$ as follows:

Let $\code=\{ X_0,\ldots, X_{N-1}\}$ be a $(v,w,\lambda_a, \lambda_c)$-OOC.
For each sequence
$X_i$, let $Q_i$ be the set of integers modulo $v$ denoting the
positions of the non-zero bits.  Then $Q_i \subseteq \Z_v$,
$|Q_i| = w$ for all $i \in \{0, \ldots N-1\}$, and
\begin{enumerate}
\item[(A$'$)] $| Q_i \cap (Q_i + \delta) | \le \lambda_a$ for all $\delta \in
  \Z_v\setminus \{0\}$.  
\item[(B$'$)] $| Q_i \cap (Q_j + \delta) | \le \lambda_c$ for all $\delta \in
  \Z_v$. 
\end{enumerate}

We give some examples of this correspondence between sets of subsets and sets of sequences:
\begin{example} \label{eg:OOC-subset}
\begin{enumerate}  
\item (\cite{ChuSalWei}) The $(13, 3, 1, 1)$-OOC, $\code = \{(1100100000000), (1010000100000)\}$ has corresponding subsets $\bie = \{ \{0, 1, 4\}, \{0, 2, 7\} \}$ over $\Z_{13}$.
\item (\cite{FujiharaMiao}) The $(15,3,1,1)$-OOC, $\code = \{(110010000000000), (100000010100000)$ has corresponding subsets $\bie = \{ \{0, 1, 4\}, \{0, 7, 9\} \}$ over $\Z_{15}$.
\end{enumerate}
\end{example}

Clearly the notions of properness and non-degeneracy of OOCs carry over to this view of OOCs as subsets of $\Z_v$ in a straightforward manner and we will not repeat them here. Viewing OOCs as subsets of $\Z_v$ also allows us to link them to some aspects of the more classical difference sets and difference families.  

We introduce here some notation that we will also use in later sections.  Let $G$ be a group of size $v$ written additively. Let $Q \subseteq G$.
We write $\Delta(Q)$ for the multiset of \textit{internal differences} in the set $Q$, 
$$ \Delta(Q) = \{x - x' \; : \; x, x' \in Q, x \neq x' \}.$$
The multiplicity $\lambda$ of an element $\delta \in \Delta(Q)$ is exactly $\lambda=|Q \cap (Q+\delta)|$, and  is the number of distinct pairs $(x, x') \in Q \times Q$, $x \neq x'$, such that $\delta=x-x'$.

If $Q, Q'$ are two distinct subsets of $G$, we write
$$\Delta(Q,Q') = \{x-x' \; : \; x \in Q, x' \in Q' \}$$ 
for the multiset of \textit{external differences}.
The multiplicity $\lambda$ of an element $\delta \in \Delta(Q,Q')$ is exactly $\lambda=|Q \cap (Q'+\delta)|$, and  is the number of distinct pairs $(x, x') \in Q \times Q'$ such that $\delta=x-x'$.

Hence we may rephrase condition (A$'$) as ``every $\delta \in
  \Z_v\setminus \{0\}$ occurs at most $\lambda_a$ times in the multiset $\Delta(Q_i)$ ($0 \leq i \leq N-1$)", and condition (B$'$) as ``every $\delta \in
  \Z_v$ occurs at most $\lambda_c$ times in the multiset $\Delta(Q_i,Q_j)$ ($0 \leq i \neq j \leq N-1$)".

We will discuss this in more detail in Section \ref{sub:diff}.  In the following sections we will view OOCs as sequences or subsets of $\Z_v$ interchangeably, depending on which view allows the clearest interpretation.

\subsection{Equivalence of OOCs} \label{sub:equivalence}

It is clear when we defined full-cyclic order and non-degeneracy (Definitions \ref{defn:fullcyclic}, \ref{defn:nondegenerate}) that some sequences are ``the same as'' others.  Here we discuss more precisely the question of equivalence.  We first consider the transformations under which the weight and the auto- and cross-correlation values of an OOC are somehow preserved. 

\begin{definition}[Auto- and cross-correlation profiles]
Let $\code=\{ X_0,\ldots, X_{N-1}\}$ be a $(v,w,\lambda_a, \lambda_c)$-OOC.
\begin{enumerate}
\item[(a)] The auto-correlation profile for $X_i = (x_t)$ is the $(v-1)$-tuple $(\mu_1, \mu_2, \ldots, \mu_{v-1})_{X_i}$, where
$$\mu_{\delta} = \sum_{t=0}^{v-1} x_t x_{t+\delta}, \; 1 \le \delta \le v-1.$$
\item[(b)] The cross-correlation profile of $X_i = (x_t)$,
$X_j = (y_t)$ is the $v$-tuple 
$(\gamma_0, \gamma_1, \gamma_2, \ldots, \gamma_{v-1})_{X_i,X_j}$, where
$$\gamma_{\delta} = \sum_{t=0}^{v-1} x_t y_{t+\delta}, \;0 \le \delta \le v-1.$$
\end{enumerate}
We omit the subscripts on the tuples if there is no ambiguity.
\end{definition}

In other words, in the auto-correlation profile $(\mu_1, \mu_2, \ldots, \mu_{v-1})_{X}$, $\mu_{\delta}$ denotes the number of collisions between sequence $X$ and its shift by $\delta$ places.  

Similarly, in the cross-correlation profile $(\gamma_0, \gamma_1, \gamma_2, \ldots, \gamma_{v-1})_{X,Y}$, $\gamma_{\delta}$ denotes the number of collisions between sequence $X$ and sequence $Y$ shifted by $\delta$ places.

It is not hard to see that the values of $\lambda_a$ and $\lambda_c$ of an OOC are unaffected by the process of translation, or shifting.  The entries in the auto-correlation and cross-correlation profiles are also shifted.  

\begin{lemma}\label{lem:shift}
Let $\code$ be a $(v,w,\lambda_a,\lambda_c)$-OOC.
\begin{enumerate}
\item If $X = (x_t) \in \code$ has auto-correlation profile $(\mu_1, \mu_2, \ldots, \mu_{v-1})$ then $X+s = (x_{t+s})$, also has auto-correlation profile $(\mu_1, \mu_2, \ldots, \mu_{v-1})$. 
\item If $X, X' \in \code$ have cross-correlation profile $(\gamma_0, \gamma_1, \gamma_2, \ldots,\gamma_{v-1})$ then $X, X'+s$ have cross-correlation profile
  $(\gamma_{0+s}, \gamma_{1+s}, \gamma_{2+s}, \ldots, \gamma_{(v-1)+s})$. 
\end{enumerate}
\end{lemma}

The proof of Lemma \ref{lem:shift} is straightforward.

Next we consider what happens to the auto- and cross-correlation profiles under an automorphism of $\Z_v$.  Let $\Z_v^* = \{r \in \Z_v : (r, v) = 1\}$ be the elements of $\Z_v$ coprime to $V$.  An automorphism of $\Z_v$ is simply a multiplication by an element of $\Z_v^*$, that is, $\phi: \Z_v \to \Z_v$ is an automorphism of $\Z_v$ if and only if $\phi$ takes the form $\phi_r(x) = rx$ for some $r \in \Z_v^*$.

\begin{lemma}\label{lem:auto}
Let $\code$ be a $(v,w,\lambda_a,\lambda_c)$-OOC. Let $r \in \Z_v^*$.
\begin{enumerate}
\item Let $X = (x_t)\in \code$ have auto-correlation profile $(\mu_1, \mu_2, \ldots, \mu_{v-1})$.  Let $X' = (x'_t)= (x_{rt})$ and write $X'=rX$.  Then $rX$  has auto-correlation profile $(\mu_r, \mu_{2r}, \ldots, \mu_{(v-1)r})$. 
\item If $X, Y \in \code$ have cross-correlation profile $(\gamma_0, \gamma_1, \gamma_2, \ldots,\gamma_{v-1})$ then $X'=rX, Y'=rY$ have cross-correlation profile
  $(\gamma_{0}, \gamma_{r}, \gamma_{2r}, \ldots, \gamma_{(v-1)r})$. 
\end{enumerate}
\end{lemma}

\begin{proof}
\begin{enumerate}
    \item Let $Q$, $Q'$ be the subset of $\Z_{v}$ corresponding to $X$ and $X'$ respectively.  We have $x'_t=1$ if and only if $x_{rt}=1$, so $t \in Q'$ if and only if $rt \in Q$.  Hence the subset corresponding to $X' = rX$ is $Q'=r^{-1}Q = \{r^{-1}q \; : \; q \in Q\}$. 
    
    Suppose $|Q' \cap (Q'+\delta)| = |r^{-1}Q \cap (r^{-1}Q+\delta)| = \mu'_{\delta}$,  that is, there are $\mu'_{\delta}$ pairs of $(g, g') \in Q \times Q$, $g \neq g'$, such that $\delta = r^{-1}g-r^{-1}g'$.  That is, $\delta r = g-g'$ for $\mu'_{\delta}$ pairs of $(g,g')$. But there are exactly $\mu_{\delta r}$ such pairs.  Hence $\mu'_{\delta} = \mu_{\delta r}$ and $rX$ has auto-correlation profile $(\mu_{r}, \mu_{2r}, \ldots, \mu_{(v-1)r})$.
    \item Similar to above.
    
\end{enumerate}
\end{proof}

Lemmas \ref{lem:shift} and \ref{lem:auto} show that shifting and applying an automorphism on the indices simply permute the entries in the auto- and cross-correlation profiles, and $\lambda_a$ and $\lambda_c$ are preserved.  This motivates the definition of equivalence of OOCs as follows. 

\begin{definition}[Equivalent OOCs] \label{defn:equivOOCs}
Let  $\code_1=\{X_0, \ldots, X_{N-1}\}$, $\code_2=\{Y_0, \ldots, Y_{N-1}\}$ be 
two non-degenerate $(v,w, \lambda_a, \lambda_c)$-OOCs.   Then $\code_1$ is equivalent
to $\code_2$ if and only if there exists an $r \in \Z_v^*$ such that for each $Y \in \code_2$, $Y = (x_{rt+s})_{t=0}^{v-1}$ for some $X = (x_t)_{t=0}^{v-1} \in \code_1$ and some $s \in \Z_v$. 
\end{definition}

Lemma \ref{lem:shift} indicates why, for OOCs, the distinction between subsets being disjoint or not is more fluid than in other combinatorial settings.  An OOC with disjoint subsets may be converted to an OOC with intersecting subsets by a series of translations, which only permute the auto- and cross-correlation profiles, leaving unchanged the values of $\lambda_a$ and $\lambda_c$.

\begin{remark}[On equivalences and isomorphisms] \label{rem:equiv}
We remark that our definition of equivalent OOCs is in line with that of difference sets (and other difference families).  A $w$-subset $D$ of a group $G$ is a difference set if, for all non-zero elements $\delta$ of $G$, $|D \cap (D+\delta)| = \lambda$ for some fixed non-zero $\lambda$.   The development of $D$ is the incidence structure whose points are the elements of $G$ and whose blocks are the translates $D+g=\{d+g \; : \; d \in D\}$, $g \in G$. 
In the setting of difference sets, 
isomorphism refers to the isomorphism of the designs resulting from the development of the difference sets.   If $D_1 \subseteq G_1$ and $D_2 \subseteq G_2$ are two difference sets in groups $G_1$, $G_2$, then we say $D_1$ and $D_2$ are isomorphic if the incidence structures resulted from their developments are isomorphic, and we say that they are equivalent if there is a group isomorphism $\phi$ between $G_1$ and $G_2$ such that $\phi(D_1)=\{\phi(d) \; : \; d \in D_1\} = D_2 + g$ for some $g \in G_2$.  Equivalent difference sets are isomorphic, but the converse is not true: there are isomorphic difference sets which are not equivalent. However, in the cyclic case (the underlying group being cyclic) all known isomorphic difference sets are also equivalent. There is a rich literature in difference sets and we refer the reader to \cite{ColbournDinitz}.  

However, in the OOC literature,  
there are different flavours of ``equivalence" and ``isomorphism".  For example, in \cite{BaichevaTopalova}, two OOCs $\code$, $\code'$ are said to be ``isomorphic" if there is a permutation of $\Z_v$ which ``maps the collection of translates of each codeword-set (codeword) of $\code$ to the collection of translates of a codeword-set of $\code'$", and that two OOCs are ``multiplier equivalent if they can be obtained from one another by an automorphism of $\Z_v$ and replacement of
codeword-sets by some of their translates.  Two OOCs can be isomorphic, but multiplier inequivalent".  This ``multiplier equivalence" corresponds to our definition of equivalence.

Here we have only defined equivalence, in line with in the classical sense.  It is not hard to see that if we consider the incidence structure arising from developing the subsets of $\Z_v$ corresponding to an OOC, then two OOCs that are equivalent are also isomorphic, and two isomorphic OOCs are equivalent.   
\end{remark}

\section{Survey of bounds and constructions}
\label{sec:boundsandconstructions}

Here we review bounds and constructions in the literature, with a focus on those that concern the case when $\lambda_a \neq \lambda_c$.

\subsection{Bounds on the size of $(v,w, \lambda_a, \lambda_c)$-OOCs}
\label{sub:bounds-on-size}

Let $\Phi(v,w, \lambda_a, \lambda_c)$ denote the size of the largest possible OOC with these parameters.

The case where $\lambda_a=\lambda_c=1$, and the case $\lambda_c=1$, have been discussed extensively, from \cite{ChuSalWei} onwards, including \cite{BaichevaTopalova,BurMomPas, BurPasWu,ChangFujiharaMiao}.  Here we focus on  $\lambda_a \neq \lambda_c$.

A well known upper bound on $\Phi(v,w, \lambda_a, \lambda_c)$  derives from the relationship between OOCs and constant weight codes (we discuss this relationship further in Section \ref{sub:codes}):

\begin{theorem}[The Johnson bound] \label{thm:Johnson}
$$\Phi(v,w, \lambda_a, \lambda_c) \le \frac{(v-1)\cdots (v-\lambda)}{w(w-1) \cdots (w-\lambda)}, \mbox{ where } \lambda = \max \{\lambda_a, \lambda_c \}.$$
\end{theorem}

In \cite[Table II]{ChungKumar} it is shown that this bound is far from tight.  It would seem that overcounting is introduced in the use of $\max \{\lambda_a, \lambda_c\}$ and in the iterative nature of the derivation of the Johnson bound.  

In \cite{YangFuja}, Yang and Fuja derived an upper bound on the size of an OOC with unequal auto- and cross-correlation values, specifically, for $\lambda_a \ge \lambda_c$.  By counting the relative delays (the number of ``steps'') between the 1s in a sequence, they obtained an upper bound on the size of a $(v, w, \lambda+m, \lambda)$-OOC:

\begin{theorem}[\cite{YangFuja}]\label{thm:YangFuja}
    $$ \Phi(v,w, \lambda+m, \lambda) \le \frac{(v-1)(v-2) \cdots (v-\lambda)(\lambda+m)}{w(w-1)(w-2) \cdots (w-\lambda)}.$$
\end{theorem}

It was noted in \cite{YangFuja} that this bound is only a generalization of the Johnson bound for $\lambda$ = 1, and for  for $\lambda \ge 2$ the bound on $\Phi(n, w, \lambda, \lambda)$ obtained by setting $m = 0$ in Theorem \ref{thm:YangFuja} is
weaker than that of Theorem \ref{thm:Johnson}.  From the proof in \cite{YangFuja} it seems that an overestimation occurs when counting all possible relative delays between the 1s.

Many tighter upper bounds are given for the case $\lambda_c=1$ - see, for example, \cite{BurMomPas,BurPasWu,WangChang,BaiTop2023}.  These bounds are derived using the structural property of the disjointness of internal differences when $\lambda_c=1$.  We discuss this in Theorem \ref{thm:intersection}.  

There are also several Gilbert-Varshamov type lower bounds in \cite{ChuSalWei,YangFuja,ZhangShangguanGe}.  We quote the two given in \cite{ChuSalWei}:

\begin{theorem}[\cite{ChuSalWei}] \label{thm:lowerGV}
$$\Phi(v, w, \lambda_a, \lambda_c) \ge 
\frac{{v \choose w} - \frac{n-1}{2} {w \choose \lambda_a +1}{v \choose w - \lambda_a -1}}{ v \sum_{i=\lambda_c+1}^{\min{\{v-w,w\}}} {v-w \choose w-i}{w \choose i}.}$$
\end{theorem}

This bound is obtained by counting the maximum number of $v$-tuples violating the auto-correlation and cross-correlation properties.

\begin{theorem}[\cite{ChuSalWei}] \label{thm:lowerChuSalWei}
$$\Phi(v, w, \lambda_a, \lambda_c) \ge
\frac{\lambda_c ( v - w + 1 ) - (\lambda_c / \lambda_a) (w -1)^2 (w-2)}{w (w-1)^2}.$$
\end{theorem}

Intuitively, the proof goes like this: suppose we already have $m-1$ codewords in our OOC, and suppose
that for the $m$th codeword we already have $w-1$ elements.  To add the next element, we try all the $x \in \Z_v$,
making sure that the auto-correlation and cross-correlation properties are not violated.  We count the number of
``bad'' elements, and if $v$ is bigger than that, then we can find one element that works.

Another lower bound using both  $\lambda_a$ and
$\lambda_c$ can be found in \cite{FujiharaMiao} - it was presented as an upper bound but should be a lower bound  because of the division
by a negative number in an inequality: $w^2 - v \lambda_c \le 0$, by Lemma \ref{lem:lcbound}.  We restate that bound
and give a proof using a simple counting argument.

\begin{theorem}
\label{thm:lowerFujiharaMiao}
For a $(v,w, \lambda_a, \lambda_c)$-OOC $\code$ with $|\code| = N \ge 2$,
$$N \ge
\frac{v(\lambda_c - \lambda_a) - w + \lambda_a}{v \lambda_c - w^2},$$
provided $w^2-v \lambda_c \neq 0$.
\end{theorem}

\begin{proof}
    Suppose $X=(x_t)$, $Y=(y_t) \in \code$. 
    Let $P(X,Y)=\sum_{d=0}^{v-1} \sum_{t=0}^{v-1} x_t y_{t+d}$.

    Consider $\sum_{X,Y \in \code} P(X,Y)$.

    Firstly,
    \begin{align*}
    \sum_{X,Y \in \code} P(X,Y) 
&= \sum_{X,Y \in \code} \left( \sum_{d=0}^{v-1} \sum_{t=0}^{v-1} x_t y_{t+d} \right) \\
&= \sum_{X \in \code} \left( \sum_{d=0}^{v-1} \sum_{t=0}^{v-1} x_t x_{t+d} \right) +\sum_{X \neq Y \in \code} \left( \sum_{d=0}^{v-1} \sum_{t=0}^{v-1} x_t y_{t+d} \right) \\
&= \sum_{X \in \code} \left( \sum_{t=0}^{v-1} x_t x_{t} \right) + \sum_{X\in \code} \left( \sum_{d=1}^{v-1} \sum_{t=0}^{v-1} x_t x_{t+d} \right) + \sum_{X \neq Y \in \code} \left( \sum_{d=0}^{v-1} \sum_{t=0}^{v-1} x_t y_{t+d} \right) \\
& \le Nw + N(v-1) \lambda_a + N(N-1)v \lambda_c.
    \end{align*}

    On the other hand, we have
    \begin{align*}
    P(X,Y) &= \left( \sum_{d=0}^{v-1} \sum_{t=0}^{v-1} x_t y_{t+d} \right) \\
&= \left( \sum_{t=0}^{v-1} \sum_{d=0}^{v-1} x_t y_{t+d} \right) \\
&= w^2,
    \end{align*}
    and so 
    $$ \sum_{X,Y \in \code} P(X,Y) = \sum_{X,Y \in \code} w^2 = N^2w^2. $$
Hence the inequality
$$ (v-1) \lambda_a + (N-1)v \lambda_c \ge Nw^2 -w.$$
Rearranging gives the lower bound on $N$, provided $w^2-v \lambda_c \neq 0$.
\end{proof}

The inequality $ (v-1) \lambda_a + (N-1)v \lambda_c \ge Nw^2 -w$ gives some evidence to the intuition that $\lambda_a$ and $\lambda_c$ trade off each other: a low $\lambda_a$ (respectively $\lambda_c$) implies a high $\lambda_c$ (respectively $\lambda_a)$.

We note that none of these lower bounds seem to be very tight.  Table \ref{table:lowerbounds} enumerates some values of the bounds for some parameters for which we know the existence of OOCs.  It is easy to see from the proof of Theorem \ref{thm:lowerFujiharaMiao} why sometimes the bounds are far from tight: we estimated
$$\sum_{X\in \code} \left( \sum_{d=1}^{v-1} \sum_{t=0}^{v-1} x_t x_{t+d} \right) \le N(v-1)\lambda_a,$$
but this is likely to be an overestimation because not all sequences have auto-correlation value reaching $\lambda_a$.  Similarly, for the next term in the inequality, not all pairs of sequences have cross-correlation value reaching $\lambda_c$.  Such over-counting seems unavoidable in this approach.  We take another approach in Section \ref{sec:bounds} and consider the bounds on $\lambda_a$ and $\lambda_c$ instead, which also results in some new constructions.

\begin{table}
\begin{center}
\begin{tabular}{|cccc|c|ccc|} \hline
$v$ & $w$ & $\lambda_a$ & $\lambda_c$ & $|\code|$ & Thm \ref{thm:lowerGV} & Thm \ref{thm:lowerChuSalWei} & Thm \ref{thm:lowerFujiharaMiao}\\ \hline
7 & 3& 1& 1& 1 \cite{ChuSalWei} & $\frac{1}{12}$ & -1 & $-\frac{1}{17}$\\
19& 3& 1& 1& 3 \cite{ChuSalWei}& $\frac{13}{12}$ & 0 & $-\frac{1}{17}$\\
43& 3& 1& 1& 7 \cite{ChuSalWei}& $\frac{37}{12}$ & 1 & $-\frac{1}{17}$\\     
8& 4& 4& 2& 3 (Eg \ref{eg:lcbound-eq}) & $\frac{1}{36}$ & 0 & $-\frac{8}{35}$\\
73& 9& 1& 3& 8$^*$ & $-\frac{383}{192}$ & -27 & $\frac{23}{8}$\\
29& 7& 2& 3& 4$^{\dag}$ & $-\frac{67}{84}$ & -6 & $-\frac{3}{10}$ \\ \hline 
 \multicolumn{8}{p{8cm}}{$^*$: Multiplicative cosets of $\langle 2 \rangle$ in $\Z_{73}$. Eg \ref{eg:labound-eq}} \\
\multicolumn{8}{p{8cm}}{$^{\dag}$: Multiplicative cosets of $\langle 25 \rangle$ in $\Z_{29}$.} \\
\end{tabular}
\caption{Some values of the lower bounds.}
\label{table:lowerbounds}
\end{center}
\end{table}

Fang and Zhou \cite{FangZhou} take another approach and 
considers orbits of $w$-subsets under the action of $\Z_v$. 
 By analysing the structures of these subsets, exact values of $\Phi(v, w, w-2, w-1)$ were obtained for $3\le w \le v$. 
 This generalises the results of \cite{HuangChang}, which gives the exact values of $\Phi(v,4,\lambda_a, 4)$, $\lambda_a=1, 2$, by using the results of \cite{MomBuratti}.

\subsection{Some construction techniques with $\lambda_a \neq \lambda_c$}
\label{sub:constructions}

We next survey constructions in the setting when $\lambda_a \neq \lambda_c$.

\subsubsection{Constructions using finite geometry}

In \cite{AldersonMellinger}, Alderson and Mellinger constructed several families of OOCs with $\lambda_a \neq \lambda_c$ using objects in finite geometry.  Note that all these have $\lambda_a < \lambda_c$, so the upper bound of Theorem \ref{thm:YangFuja}, which deals with the case $\lambda_c \le \lambda_a$, does not apply.  The constructions include:

\begin{itemize}
    \item $\left( \frac{q^{k+2}-1}{q-1}, q+1, k, k+2 \right)$-OOCs, with approximately $q^{k^2+2k-3}$ codewords, from a $t$-family of $m$-arcs in $PG(k,q)$, which are $m$-arcs pairwise meeting in at most $t$ points.  (Here $t=k+2$.)
    \item $\left( \frac{q^{k+2}-1}{q-1}, q, k, k+1 \right)$-OOCs, with approximately $q^{2k^2+3k-2}$ codewords, from arcs in Baer subspaces of $PG(k,q^2)$.
    \item $\left( \frac{q^{2(k+2)}-1}{q^2-1}, \frac{q^{k+1}-1}{q-1}, \frac{q^{k}-1}{q-1}, \frac{q^{k}-1}{q-1}+1 \right)$-OOCs, from Baer subspaces in $PG(k,q^2)$.
\end{itemize}

\subsubsection{Constructions using number theory}
There are many constructions based on number theoretic properties and cyclotomy.

In \cite{MomBuratti}, Momihara and Buratti obtained upper bounds for OOCs with parameters $(v, 3,2,1)$, $(v, 4, 3,1)$ and $(v, 4, 2,1)$ by examining the relative delays and analysing the correspondence between the patterns of delays and the structures of the subsets.  A few constructions were given in  \cite{MomBuratti}, some of which met these bounds.  We list a few here:

\begin{itemize}
    \item $(p,4,2,1)$-OOCs, $p$ prime, $p \equiv 1 \pmod{8}$, $-4$ is not a $2^{e}$th power in $\Z_p$.  The OOCs are cosets of the fourth roots of unity in $\Z_p$.
    \item $(2p, 4, 2, 1)$-OOCs for prime $p \equiv 1 \pmod{4}$.
\end{itemize}

In \cite{BurMomPas}, Buratti et al. built on the work of \cite{MomBuratti} and gave nonexistence results for infinitely many values of $v$, and some constructions of $(v,4, 2, 1)$-OOCs.  Some of these results are obtained through case-by-case analysis of the number of distinct internal differences, and some of the constructions take the same approaches as \cite{MomBuratti}.  The constructions of \cite{WangChang} by Wang and Chang also take the same approach.

In \cite{BurPasWu} Buratti et al. constructed $(v,5,2,1)$-OOCs, again with a similar approach, with dependencies on the Theorem of Weil on multiplicative character sums, and graph labellings. 

We note that in the derivation of upper bounds and the constructions of OOCs with $\lambda_c=1$, much depended on the fact sets of internal differences are disjoint if $\lambda_c=1$. 

\subsubsection{Constructions using binary sequences}
It can be shown (see Lemma \ref{lem:lcbound}) that a lower bound for $\lambda_c$ is given by $\frac{w^2}{v}$.  In \cite{HuczynskaNg}, the current authors used binary sequences to construct new families of OOCs with this minimum $\lambda_c$.

\begin{theorem}
\label{thm:lcbound-eq} 
\begin{enumerate}
    \item[(a)] \cite[Theorem 3.1]{HuczynskaNg} If $v|w^2$ and $w|v$ there is an infinite family of $(v,w,\lambda_a, \lambda_c)$-OOC $\{X, Y\}$, with $\lambda_a=w$ and $\lambda_c=\frac{w^2}{v}$: 
    \begin{align*}
    X &= \overbrace{11\ldots1}^w \overbrace{00\ldots0}^w \ldots \overbrace{00\cdots0}^w, \\
    Y &= \underbrace{\underbrace{1\ldots1}_{\lambda_c}0\ldots0}_w \underbrace{\underbrace{1\ldots1}_{\lambda_c}0\ldots0}_w \ldots \underbrace{\underbrace{1\ldots1}_{\lambda_c}0\ldots0}_w.
    \end{align*}
    Written as subsets of $\Z_v$:
    \begin{align*}
        X & = \{0, 1, 2, \ldots, w-1\}, \\
        Y & = \left\{\delta w, \delta w+1, \ldots, \delta w+ \lambda_c-1 \; | \; \delta = 0, 1, \ldots, \left(\frac{v}{w}-1\right) \right\}.
    \end{align*}
    \item[(b)] \cite[Theorem 4.2]{HuczynskaNg} Appending a zero to both $X$ and $Y$ of the OOC above will give a $(v+1, w, \lambda_a, \lambda_c)$-OOC with $\lambda_c = \lceil w^2/(v+1) \rceil$ and $\lambda_a = w-1$.
    
    \item [(c)] \cite[Theorem 3.4]{HuczynskaNg} Let $N>1$.  There exists a $(2^N, 2^{N-1}, 2^{N-1}, 2^{N-2})$-OOC $\{X_1, X_2, \ldots, X_N\}$ of size $N$, where
$X_i=(x_t)_{t=0}^{v-1}$ is defined as follows:
\begin{align*}
    & x_0 = \cdots =x_{2^{i-1}-1} = 1, \\
    & x_{2^{i-1}} = \cdots = x_{2^{i}-1} = 0, \\
    & x_t = x_{t+2^i}, t \ge 2^i.
\end{align*}
So for each $X_i$ we have:
$$X_i = \underbrace{1\ldots1}_{2^{i-1}}\underbrace{0\ldots0}_{2^{i-1}} \underbrace{11\ldots100\ldots0}_{2^i} \ldots \underbrace{11\ldots100\ldots0}_{2^i}. $$

\end{enumerate}
\end{theorem} 

\begin{example} \label{eg:lcbound-eq}
\begin{itemize}
\item[(i)] For (a), $\{111000000, 100100100\}$ is a $(9,3,3,1)$-OOC.  
\item[(ii)] For (b), $\{1110000000, 1001001000\}$ is a $(10,3,2,1)$-OOC.
\item[(iii)] For (c), $\{11110000, 11001100, 10101010\}$ is an $(8,4,4,2)$-OOC. 
\end{itemize}
\end{example}

\subsubsection{Computer searches}
In \cite{ZhangPengYang}, $(v,4,2,1)$-OOCs were obtained via a computer search. A graph is constructed with vertex set codewords satisfying the auto-correlation constraint, and edges are added between vertices/codewords satisfying the cross-correlation constraint.  The search for a maximal OOC is thus transformed to the maximum clique problem.    The computer search of Baicheva and Topalova in \cite{BaichevaTopalova} is equivalent to this search, where largest possible $(v,4,2,1)$-OOCs were found for $v \le 181$. Another computer search \cite{BaiTop2023} gives  maximal size of $(v,6,2,1)$- and $(v,7,2,1)$-OOCs for $v \le 165$ and $v \le 153$, and upper bounds for larger $v$.

In \cite{Li}, Li gave an algorithm for constructing general $(v,w, \lambda_a, \lambda_c)$-OOC via block designs.  Codewords are produced from difference sets, and then more codewords are produced and accepted if the cross-correlation values satisfy the constraints.  

\subsubsection{Recursive constructions}

In \cite[Theorem 5.3]{MomBuratti}, a $(gmn, 4, 2, 1)$-OOC is constructed from a $(gm, 4, 2, 1)$-OOC and a $(gn, 4, 2, 1)$-OOC using difference matrices associated with a class of graphs called kites.

The construction in \cite[Theorem 5.1]{BurMomPas}  combines two OOCs and depends on a graph defined on units of $\Z_v$, akin to a Cayley graph, and the existence of a perfect matching. 

\subsubsection{Some constructions with $\lambda_a=\lambda_c >1$}

We look briefly at some construction techniques which result in OOCs with $\lambda_a=\lambda_c$, but paying particular attention to those with values greater than 1.

There are many constructions using objects in finite geometry in this class of construction.  For example, Alderson and Mellinger \cite{AldersonMellinger} constructed  $(\frac{q^{k+2}-1}{q-1}, q, k, k)$-OOCs from $m$-arcs in $PG(k,q)$, a generalisation of the $(q^3+q^2+q+1, q+1, 2, 2)$-OOCs constructed from conics in $PG(2,q)$ by Miyamoto et. al \cite{MiyMizShi}.  Chung et. al \cite{ChuSalWei} constructed  $(\frac{q^{d+1}-1}{q-1}, \frac{q^{s+1}-1}{q-1}, \frac{q^{s}-1}{q-1},\frac{q^{s}-1}{q-1})$-OOCs from $s$-dimensional subspaces in $PG(d,q)$, analogous to the affine version of $(q^a-1, q^d, q^{d-1}, q^{d-1})$-OOCs from $d$-flats in $EG(a, q)$. There are also $(\frac{q^{k(d+2)}-1}{q^k-1},q-1, d, d)$-OOCs from $(q-1)$-arcs in Baer subplaces, and sublines of $PG(2,q)$ in $PG(2, q^k)$, for $k$, $d >1$, in \cite{AldMel2}.

As above, number theory and cyclotomy also yield constructions here.  For example, Chung and Kumar \cite{ChungKumar} constructed $(p^{2m}-1, p^m+1, 2, 2)$-OOCs from the discrete logarithm of solutions to a degree $p+1$ equation in $\F_{p^{2m}}$, and Ding and Xing \cite{DingXing} constructed 
 $(2^m-1, w, 2,2)$-OOCs from cyclotomic classes.

\section{Relationship with other combinatorial objects}
\label{sec:others}

\subsection{Constant weight codes and cyclically permutable codes}
\label{sub:codes}

 There is a natural connection between non-degenerate OOCs and constant weight codes and cyclically permutable codes.  We clarify the relationships here, especially between $\lambda_a$, $\lambda_c$ and the weight/distance of these codes.  
 
 \begin{definition}[Constant weight cyclically permutable code \cite{Gilbert, BitanEtzion, NguyenGyorfiMassey, Neumann}]
A (binary) constant weight cyclically permutable code (CW-CPC)
of length $v$, weight $w$, minimum Hamming distance $d_H$, is a set of
binary $v$-tuples (codewords), each of weight $w$, such that $d_H$ is the smallest of the Hamming distance between  pairs of codewords, and
\begin{enumerate}
\item[(a)] all codewords are cyclically distinct, that is, no codeword can be obtained from another by cyclic shifts;
\item[(b)] each codeword is of full cyclic order, that is, a codeword is not a cyclic shift of itself unless it is a shift of $n$ (and multiples of
$n$) positions.
\end{enumerate}
We write this as $\CW(v, w, d_H)$.
\end{definition}

Note that here Hamming distance between two codewords is the number of places where the two codewords differ.  In \cite{Gilbert, Neumann, NguyenGyorfiMassey} another type of distance associated with cyclically permutable codes is defined:  

\begin{definition}[Cyclic minimum distance \cite{NguyenGyorfiMassey}]
The cyclic minimum distance $d_c$ of a cyclically permutable code is the minimum Hamming distance from a codeword to one of its own distinct shift or to some cyclic shift of another codeword.
\end{definition}

For constant weight codes, $d_c$ and $d_H$ must be even, and $d_c \le d_H$.  By treating each sequence of a $(v,w, \lambda_a, \lambda_c)$-$\OOC$ as a binary codeword, an $\OOC$ can be regarded as a CW-CPC, and vice versa. The next theorem then follows.

\begin{theorem} \label{thm:OOC-CWCPC}
\begin{enumerate}
    \item[(a)] A non-degenerate proper $(v,w, \lambda_a, \lambda_c)$-$\OOC$ gives rise to a $\CW(v,w,d_H)$ with cyclic minimum distance $d_c$, where $d_H \le 2(w-\lambda_c)$ and $d_c \le 2(w-\lambda)$, $\lambda=\max\{\lambda_a, \lambda_c\}$.
    \item[(b)] A $\CW(v, w, d_H)$ with minimum cyclic distance $d_c$ gives rise to a $(v, w, \lambda_c, \lambda_a)$-$\OOC$ with $\lambda_a, \lambda_c \le  w-\frac{d_c}{2}$, with equality for $\lambda_a$ or $\lambda_c$.
\end{enumerate}
\end{theorem}
\begin{proof}
\begin{enumerate}
    \item[(a)] Suppose $\code$ is a non-degenerate proper $(v,w,\lambda_a, \lambda_c)$-$\OOC$. Since $\code$ is proper there exist $X$, $Y \in \code$ with cross-correlation value $\lambda_c$ for some shift of $Y$, say $Y+s$.  Then $\code'=\code \setminus \{Y\} \cup \{Y+s\}$ is our $\CW$ with the specified parameters, since for every pair $X', Y' \in \code'$, there are at most $\lambda_c$ positions where $X'$ and $Y'$ both have a 1, and therefore at least $2(w-\lambda_c)$ places where they differ, so the minimum Hamming distance of $\code'$ is $d_H=2(w-\lambda_c)$.  

    For cyclic distance, similar arguments imply that the distance between a codeword and its own shift is at least $2(w-\lambda_a)$ (attained for some codeword and its shift since $\code$ is proper), and between a codeword and a shift of another codeword is at least $2(w-\lambda_c)$. Hence $d_c = 2(w-\lambda)$, $\lambda=\max\{\lambda_a, \lambda_c\}$.
    \item[(b)] If we take two codewords $X$, $Y$ of the $\CW$,  any shift of $Y$ collides with $X$ in at most $w-\frac{d_c}{2}$ places. So $\lambda_c \le w-\frac{d_c}{2}$.  Similarly, $X$ collides with any of its own shifts in at most $w-\frac{d_c}{2}$ places, so $\lambda_a \le w-\frac{d_c}{2}$.  By definition $d_c$ must be attained, so equality is achieved by either $\lambda_a$ or $\lambda_c$.  
\end{enumerate}    
\end{proof}

The well-known upper bound - the Johnson bound (quoted in Theorem \ref{thm:Johnson}) - is a consequence of identifying a proper,
non-degenerate $(v,w,\lambda_a, \lambda_c)$-OOC $\code$ with a
constant weight code: Treat the $\OOC$ as a $\CW$, and take cyclic closure of all the codewords, that is, taking all the cyclic shifts of the
codewords, to obtain a constant weight code $\code'$ with $v |\code|$
codewords and minimum distance $2w-2 \lambda$ with $\lambda =
\max\{\lambda_a, \lambda_c\}$.  Then the Johnson bound on constant
weight codes applies and we obtain an upper bound for the largest
possible size of $(v,w, \lambda, \lambda)$-OOCs.

For precision we note here that the $(v,w,\lambda_a, \lambda_c)$-OOC
has to be a proper, non-degenerate OOC:
the conclusion that $\code'$ has $v |\code|$ codewords assumes
that the OOC is non-degenerate, and the
conclusion that $\code'$ has minimum distance $2(w-\lambda)$
assumes that the OOC is proper.  

We summarise this in a lemma:

\begin{lemma} \label{lem:OOC-CW}
A non-degenerate proper $(v,w, \lambda_a, \lambda_c)$-OOC gives rise to 
a constant weight code of length $v$, weight $w$ and minimum distance
$2w-2 \lambda$, where $\lambda = \max\{\lambda_a, \lambda_c\}$.
\end{lemma}

We leave the reader with an example.

\begin{example} \label{eg:OOC-CWCPC}
Consider the set of two sequences $\{11100010010, \, 00011101001\}$.
This is a proper, non-degenerate $\OOC$ with length $v=11$, weight $w=5$, and $\lambda_a=2$, $\lambda_c=3$. 
Taking the sequences as codewords this gives a $\CW$ with cyclic minimum distance $d_c=2(w-\lambda_c)=4$.  Taking cyclic closure of the codewords we have a constant weight code with minimum distance $d_H=4$.
\end{example}

\subsection{Conflict avoiding codes and other codes}
\label{sub:CAC}

Conflict-avoiding codes (CACs) are defined, for example, in  \cite{MomCAC, ShumWongChen},  as follows:

\begin{definition}
A set $\code = \{X_0, X_1, \ldots, X_{N-1} \}$ of $N$ 
binary sequences of weight $w$ and length $v$ is called a 
conflict-avoiding code (CAC) if $\Delta(X_i) \cap \Delta(X_j)=\emptyset$ for all $i \neq j$, $X_i, X_j \in \code$.
\end{definition}

Recall the notation introduced at the end of Section \ref{sub:OOCsubsets}: $\Delta(X)$ refers to the multiset of internal differences in $X$.

CACs can be thought of as OOCs with $\lambda_c=1$ (Theorem \ref{thm:intersection}) and no consideration of the auto-correlation value $\lambda_a$, that is, a CAC is a $(v, w, w, 1)$-OOC, which is c-proper but not necessarily a-proper.  This also means that, since the auto-correlation value is not a constraint, one can obtain larger sets of CACs than OOCs with $\lambda_c=1$.  For example, an upper bound for the size of an  $(v, 4, 2, 1)$-OOC is $\lceil v/8 \rceil$ (\cite{BurMomPas}) while one for a CAC of length $v$, weight $4$, is $(v-1)/6$ (\cite{ShumWongChen}, subject to some number theoretic constraints).   

Another class of sequences that can be thought of as OOCs is the impulse radio sequences (IRSs) \cite{IRS}: these are binary sequences that satisfy the same auto- and cross-correlation properties as OOCs, with the addition of the condition that the 1s of the sequences are places in a particular way, satisfying the ``pulse position property": for a sequence $X=(x_i)$ of length $km$, the 1s are in positions $\{a_j + jm: j \in \{0, \ldots, k-1\} \}$.  So the sequence is divided into k blocks of length m, and there is a 1 in each block.  An IRS of length $km$ has weight $k$ and can be thought of as a $(km, k, \lambda_a, \lambda_c)$-OOC.  In \cite{IRS} the authors studied IRSs with $\lambda_c=k-1$, which is maximum for sequences with full cyclic order anyway, and therefore their effort focuses on $\lambda_a$.  They obtained the values of maximum sizes of IRS for parameters $k=3, 4$, and the technique used is that of considering families of distinct representatives of orbits of $k$-subsets of $\Z_{km}$ under the action of $\Z_{km}$, similar to that mentioned at the end of Section \ref{sub:bounds-on-size}.

\subsection{Difference families and variants}
\label{sub:diff}

Earlier, it was shown how OOCs can be equivalently expressed in terms of multisets of differences between subsets of $\mathbb{Z}_n$.  Difference families and their variants in $\mathbb{Z}_v$ have a natural link to OOCs in the case when $\lambda_a=\lambda_c=\lambda=1$.  So far, this link does not seem to have been explored for the situations when $\lambda>1$ or $\lambda_a \neq \lambda_c$.  This is partly because there is not such a simple connection.  We examine some of these connections and some possible generalisations here.

\begin{definition}\label{defn:df}
Let $G$ be a group of order $v$.  A collection $\{Q_0, \ldots, Q_{N-1}\}$ of $w$-subsets of $G$ forms a $(v,w,\lambda)$-difference family (DF) if every nonidentity element of $G$ occurs $\lambda$ times in $\cup_{i=0}^{N-1} \Delta(Q_i)$.  If $\{Q_0, \ldots, Q_{N-1}\}$ are pairwise disjoint then this is called a disjoint difference family.
\end{definition}

Note that the existence of a $(v, w, \lambda)$-DF necessarily implies that $\lambda (v-1) = N w (w-1)$, that is, $N=\frac{\lambda (v-1)}{w (w-1)}$.  A $(v,w,1)$-DF in $\mathbb{Z}_v$ may be regarded as a $(v,w,1,1)$-OOC of size $N=\frac{(v-1)}{w (w-1)}$.  This is observed, for example, in \cite{Bur}.  To see this, the following lemma is required.

\begin{lemma} \label{lem:disjoint}
Let $B=\{Q_0, \ldots, Q_{N-1}\}$ be the subsets of $\Z_v$ corresponding to a $(v,w,\lambda_a, \lambda_c)$-$\OOC$.  Then $\lambda_c=1$ if and only if the multisets $\{ \Delta(Q_0), \ldots \Delta(Q_{N-1})\}$ are pairwise disjoint.
\end{lemma}
\begin{proof}

Let $\lambda_c=1$.  Suppose there is some $\delta \in \mathbb{Z}_v^*$ such that $\delta \in \Delta(Q_i) \cap \Delta(Q_j)$ where $i \neq j$.  So there exist distinct $x_1,x_2 \in Q_i$ and distinct $y_1, y_2 \in Q_j$ such that $\delta=x_1-x_2=y_1-y_2$.  Rearranging, $x_1-y_1=x_2-y_2(=\beta)$ say, so there are two distinct expressions for $\beta \in \mathbb{Z}_v^*$ in $\Delta(Q_i,Q_j)$, contradicting $\lambda_c=1$.  The converse is immediate.
\end{proof}

It is then straightforward to show:
\begin{proposition} \label{prop:df-OOC}
If $B=\{Q_0, \ldots, Q_{N-1}\}$ is a $(v,w,1)$-DF in $\mathbb{Z}_v$, then $B$ is a $(v,w,1,1)$-$\OOC$ of size $N=\frac{v-1}{w(w-1)}$.
\end{proposition}
\begin{proof}
By definition, the multiset $\cup_{i=0}^{N-1} \Delta(Q_i)$ contains every element of $\mathbb{Z}_v^*$ precisely once. Thus the $\Delta(Q_i)$ $(0 \leq i \leq N-1)$ are disjoint, and so by Lemma \ref{lem:disjoint} we have $\lambda_c=1$. Moreover any $\delta \in \mathbb{Z}_v^*$ occurs as a difference in each $\Delta(Q_i)$ at most once, so $\lambda_a=1$. Finally, the relation $\lambda (v-1)=N w(w-1)$ holds for $B$, and so $N=\frac{v-1}{w(w-1)}$.
\end{proof}

For the converse, the notion of a \emph{scarce difference family} is introduced in \cite{Yin} as follows: 
\begin{definition}\label{SDFdef}
A $(v,w,1)$-scarce difference family (SDF) is a collection $S$ of $w$-subsets of $\mathbb{Z}_v$, such that for every non-zero $x \in \mathbb{Z}_v$, the congruence $d_i-d_j \equiv x \pmod{v}$ has at most one solution pair $(d_i, d_j)$ with $d_i, d_j \in B$ for some $B \in S$.
\end{definition}

The following is immediate (see \cite{Yin}):
\begin{theorem} \label{thm:sdf-ooc}
A $(v,w,1)$-SDF with $N$ blocks is equivalent to a $(v,w,1,1)$-$\OOC$.
\end{theorem}

One way to obtain an SDF is to take a relative difference family with frequency $1$.  Relative difference families were introduced in \cite{BurRDF} and defined as follows:
\begin{definition}[Relative difference family] \label{defn:RDF}
Let $G$ be a group of order $v$ with subgroup $H$ of order $n$. A $(v,n,w,\lambda)$-DF over $G$ relative to $N$ is a collection $\{Q_0,\ldots, Q_{N-1}\}$ of $w$-subsets of $G$ such that the multiset union $\cup_{i=0}^{N-1} \Delta(Q_i)$ comprises $\lambda$ copies of elements of $G \setminus H$ and zero copies of $H$.
\end{definition}

A $(v,n,w,1)$-DF over $\mathbb{Z}_v$ relative to any subgroup $H$ is a $(v,w,1)$-SDF and hence a $(v,w,1,1)$-$\OOC$ \cite{Yin}.  These have been much-studied and many constructions are known.

We propose generalising the definition of SDFs to larger $\lambda$ as follows. 

\begin{definition} \label{defn:sdf-lambda}
Let $B=\{Q_0, \ldots, Q_{N-1}\}$ be a collection of $w$-subsets of $\mathbb{Z}_v$.  Then we say $B$ is a $(v,w,\lambda)$-SDF if for each $x \in \mathbb{Z}_v^*$, there is at most one $i \in \{0,\ldots, N-1\}$ such that $x \in \Delta(Q_i)$, and if $x \in \Delta(Q_i)$ then $x$ occurs at most $\lambda$ times in $\Delta(Q_i)$. We call this a proper SDF if $\lambda$ is achieved by at least one $x$.
\end{definition}

\begin{theorem}
A (proper) $(v,w,\lambda)$-SDF with $N$ blocks is equivalent to a (proper) $(v,w,\lambda,1)$-$\OOC$.
\end{theorem}
\begin{proof}
Let $B=\{Q_0, \ldots, Q_{N-1}\}$ be a $(v,w,\lambda)$-SDF. By definition, $\{ \Delta(Q_0), \ldots, \Delta(Q_{N-1})\}$ are disjoint so $\lambda_c=1$ by Lemma \ref{lem:disjoint}.  Each element of $\mathbb{Z}_v^*$ occurs at most $\lambda$ times in any $\Delta(Q_i)$ and so we can take $\lambda_a=\lambda$.  The converse is immediate.
\end{proof}

We note that this gives us a link with conflict-avoiding codes (Section \ref{sub:CAC}), whose definition requires $\Delta(Q_i) \cap \Delta(Q_j)=\emptyset$ for all $i \neq j$ but has no auto-correlation property - so a $(v,w,\lambda)$-SDF with any $\lambda$ is a conflict-avoiding code. Note also that here $\lambda_c=1$ and many results depend on the disjointness of the internal differences (Lemma \ref{lem:disjoint}). 

External difference families (EDFs) and their generalisations have been much-studied (see \cite{PatersonStinson}) and it is natural to ask about connections between OOCs and EDFs. 
\begin{definition}
Let $G$ be a group of order $v$.  A $(v, N, w, \lambda)$-EDF is a collection of $N$ disjoint $w$-subsets $\{Q_0, \ldots, Q_{N-1} \}$ of $G$ such that every non-identity element of $G$ occurs exactly $\lambda$ times in the multiset $\cup_{i \neq j} \Delta(Q_i, Q_j)$. 
\end{definition}

From this definition it is clear that, despite their similarities, it is far from easy to obtain OOCs from EDFs, since knowing the parameters of an EDF does not give an indication of the number of occurrences of elements as differences in each individual $\Delta(Q_i,Q_j)$. Moreover, unlike OOCs, EDFs require disjoint sets.  We can, however, say more about certain variants.  

\begin{definition}
Let $G$ be a group of order $v$.  A $(v, N, w; \lambda)$-strong external difference family (SEDF) is a collection of $N$ disjoint $w$-subsets $\{Q_0, \ldots, Q_{N-1} \}$ of $G$ such that for every $i \in \{0, 1, \ldots, N-1\}$, every non-identity element of $G$ occurs exactly $\lambda$ times in $\cup_{\{j: j \neq i\}} \Delta(Q_i, Q_j)$.
\end{definition}
We have the following:
\begin{theorem}\label{SEDF}
A $(v,N,w,\lambda)$-SEDF $\{Q_0,\ldots, Q_{N-1}\}$ in $\mathbb{Z}_v$ is a $(v,w,\lambda_a,\lambda)$-$\OOC$ where $\lambda_a$ is the maximum occurrence of any non-zero element in $\Delta(Q_i)$, for all $1 \leq i \leq N$.
\end{theorem}
When a $(v,N,w,\lambda)$-SEDF partitions the whole group (respectively, its set of non-identity elements), it is known \cite{Wenetal} that each $Q_i$ must be a $(v,w,w-\lambda)$-difference set (respectively a $(v,w,w-\lambda-1,w-\lambda)$-partial difference set) and so in this case $\lambda_a = w-\lambda$ and we obtain a $(v,w, w-\lambda, \lambda)$-OOC from Theorem \ref{SEDF}.  

\begin{example}  \label{ex:SEDF}
The sets $\{1,3,4,9,10,12\}$ and $\{2,5,6,7,8,11\}$ in $\mathbb{Z}_{13}$ form a $(13,2,6,3)$-SEDF and hence yield a $(13,6,3,3)$-OOC. This is a special case of an infinite family of cyclotomic 2-set SEDFs: for any prime $p$ such that $p \equiv 1 \mod 4$, there exists a $(p,2,\frac{p-1}{2}, \frac{p-1}{4})$-SEDF in $\mathbb{Z}_p$ whose sets are Paley partial difference sets comprising the nonzero squares and nonsquares in $GF(p)$.  These yield $(p, \frac{p-1}{2}, \frac{p-1}{4}, \frac{p-1}{4})$-OOCs. At present, one example of an SEDF in an abelian group with $N>2$ is known; all others have $N=2$ \cite{JedwabLi}.
\end{example}

If we relax the requirement for disjointness of the sets, we can obtain non-disjoint SEDFs from OOCs \cite{HuczynskaNg}.

Finally, particular constructions of EDF-like objects can be used to create OOCs, if they possess appropriate structural properties.  As an example, we show how a recent construction in \cite{HuczynskaJohnson} for disjoint partial difference families and external partial difference families (variants of disjoint difference families and external difference families where elements occur at one of two frequencies) can be adapted to yield an OOC with $\lambda_a \neq \lambda_c$.

\begin{theorem}
Let $m>3$ be an odd integer.  In $\mathbb{Z}_{2m}$, for each $1 \leq i \leq 2m-1$, define the subsets
$$ T_i=\{0,i,m-i,m+i,2m-i\}.$$
Then the collection of sets $T=\{T_1, \ldots, T_{\frac{m-1}{2}}\}$ yields a $(2m,5;4,3)$-OOC.
\end{theorem}
\begin{proof}
In Proposition 3.1 of \cite{HuczynskaJohnson}, the subsets $S_i=T_i \setminus \{0\}=\{i,m-i,m+i,2m-i\}$ of $\mathbb{Z}_{2m}$ are defined ($1 \leq i \leq 2m-1$).  It is shown that $S_i=S_{m-i}=S_{m+i}=S_{2m-i}$, that $S_1, \ldots, S_{\frac{m-1}{2}}$ comprise all distinct $S_i$ in $\mathbb{Z}_{2m}$ and partition $\mathbb{Z}_{2m} \setminus \{0,m\}$, and that for $1 \leq i \neq j \leq \frac{m-1}{2}$, $\Delta(S_i)=4\{m\} \cup 2S_{2i}$ while $\Delta(S_i,S_j)=2S_{i-j} \cup 2S_{i+j}$.  

It is clear that, for $1 \leq i \neq j \leq \frac{m-1}{2}$, $\Delta(T_i)=4\{m\} \cup 2S_i \cup 2S_{2i}$ (since $S_i=-S_i$ by construction) and $\Delta(T_i,T_j)=\{0\} \cup S_i \cup S_j \cup 2S_{i-j} \cup 2S_{i+j}$.  Hence each nonzero element of $\mathbb{Z}_{2m}$ occurs at most $4$ times in any $\Delta(T_i)$ ($1 \leq i \leq \frac{m-1}{2}$), so the OOC corresponding to $T$ has $\lambda_a=4$.  For $1 \leq i \neq j \leq \frac{m-1}{2}$, it is clear that $S_i$ and $S_j$ are distinct; moreover it can be verified that $S_{i-j}$ and $S_{i+j}$ are distinct (since, modulo $2m$, $i+j$ is not equal to any of $i-j,m-(i-j),m+(i-j)$ or $2m-(i-j)$).  Hence each element of $\mathbb{Z}_{2m}$ occurs at most $3$ times in any $\Delta(T_i,T_j)$ and so $\lambda_c=3$ for the corresponding OOC.
\end{proof}

\begin{example}
\begin{itemize}
\item[(i)] The sets $\{ \{0, 1, 4, 6, 9\}, \{0,2, 3, 7, 8\} \}$ in $\mathbb{Z}_{10}$ yield a $(10,5,4,3)$-OOC.
\item[(ii)] The sets $\{ \{0, 1, 8, 10, 17\}, \{0, 2, 7, 11, 16\}, \{0, 3, 6, 12, 15\}, \{0, 4, 5, 13, 14\} \}$ in $\mathbb{Z}_{18}$ yield an $(18,5,4,3)$-OOC.
\end{itemize}
\end{example}

\subsection{Packings}
\label{sub:packing}

Packings give OOCs with $\lambda_a = \lambda_c \ge 1$.  

\begin{definition}[$t$-$(v,w, \lambda)$ packing] \label{defn:packing}
A $t$-$(v,w, \lambda)$ packing (or packing design) $(V, \bie)$ consists of a set $V$ of size $v$, a collection $\bie$ of $w$-subsets of $V$ called blocks, such that every set of $t$ elements of $V$ is contained in at most $\lambda$ blocks.
\end{definition}

A $t$-$(v, w, \lambda)$ packing 
admitting a cyclic and point-regular automorphism group is called a cyclic $t$-$(v, w, \lambda)$ packing.  The point set $V$ can be identified with $\Z_v$ and the automorphism is then $\sigma: i \to i+1 \pmod{v}$.  If $B=\{b_0, b_1, \ldots, b_{w-1}\}$ is a block then the orbit of $B$ is the set of blocks $\{B^{\sigma^i}=\{b_0+i, b_1+i \ldots, b_{w-1}+i \} \; | \; i \in \Z_v\}$.  If a block orbit have $v$ distinct blocks we say it is a full orbit, otherwise it is short.  We may choose an arbitrarily fixed block from each 
block orbit and call it a base block.

Different versions of the following theorem are stated and proved in \cite{FujiharaMiao, Yin}.  We state it in its most general form here and include a proof.
 
\begin{theorem} \label{thm:packing}
A cyclic $t$-$(v,w, 1)$-packing with $N$ full (block) orbits is
equivalent to a non-degenerate $(v, w, t-1, t-1)$-OOC of size $N$.
\end{theorem}

\begin{proof}
Let $\design=(V, \bie)$ be a cyclic $t$-$(v,w, 1)$-packing with
$N$ full (block) orbits, and let $\design'$ be a collection of $N$ base blocks $\{B_0, \ldots, B_{N-1}\}$ that are full blocks.  Consider $|B_i \cap B_j|$.
Suppose $|B_i \cap B_j| \ge t$.  Then a set of $t$ points of $V$ is
contained in 2 blocks, and this contradicts the definition of a
$t$-$(v,w, 1)$-packing.  Hence $|B_i \cap B_j| \le t-1$, and
this applies to all sets $B_i$, $B_j$ and their translates.  This means that $\code=\{B_0, \ldots, B_{N-1}\}$ is a
$(v, w, t-1, t-1)$-OOC of size $N$. 

Let $\code$ be a non-degenerate $(v, w, t-1, t-1)$-OOC of size $N$ with
corresponding subsets $\{B_0, \ldots, B_{N-1}\}$ of $\Z_v$.  Let
$\bar{\code}$ be the cyclic closure of $\code$ and let $\design$ be
the corresponding set system consisting of $\{B_0, \ldots, B_{N-1}\}$
and all their translates.  Then $\design$ has $v$ points and $vN$
blocks partitioned into $N$ full orbits. Every set of $t$ points are
contained in at most one block: if there is a set of $t$ points
contained in two blocks then the corresponding sequences will have
correlation value $t$.  Hence $\design$ is a cyclic $t$-$(v,w,
1)$-packing design with $N$ full (block) orbits.
\end{proof}

Note that the proof does not work at all for any $t$-$(v,w,\lambda)$-packing for $\lambda > 1$, and the proof does not distinguish between $\lambda_a$
and $\lambda_c$.

\begin{note}[On ``optimum'' (meeting the bound) and ``maximum'' (biggest possible)] 
\cite{Yin} says that a cyclic
  $2$-$(v,k,1)$-packing design has either $k(k-1) | v-1$ and each orbit is full,
or $k(k-1) | v-k$ and all orbits apart from one is full.  So we have an ``optimal'' cyclic $2$-$(v,k,1)$-packing design that gives an ``optimal'' OOC in the first case, and an ``optimum'' $2$-$(v,k,1)$-packing design that gives a ``maximum'' OOC in the second, with $\lfloor (v-1)/k(k-1) \rfloor$ sequences.  Example 2.2 in \cite{Yin} has $v=15$, $k=3$.  The OOC is maximum, the design has 3 orbits, one with 5 blocks and 2 with 15 blocks. 
\end{note}

\section{Bounds on $\lambda_a$ and $\lambda_c$}
\label{sec:bounds}

We focus here on the bounds on the values of $\lambda_a$ and $\lambda_c$, and consider what structural insights might be found in the proofs of these bounds.  The published bounds on the the size $N$ of an OOC are discussed in Section \ref{sub:bounds-on-size} but are not our chief concern here.

\subsection{Bounds on $\lambda_a$}
\label{sub:lambda_a_bounds}

As mentioned in the Introduction, $\lambda_a$ mainly relates to the issue of synchronisation and minimising does not seem to be the priority in practice.  Here we examine the constraints on $\lambda_a$ in a $(v, w, \lambda_a, \lambda_c)$-OOC, given $v$ and $w$.  We consider first what the lower bound is.

\begin{lemma} \label{lem:labound}
Let $\code = \{ X_0, \ldots, X_{N-1}\}$ be a $(v,w, \lambda_a, \lambda_c)$-OOC.  Then 
\[ \lambda_a \ge \frac{w(w-1)}{v-1}. \]
\end{lemma}  
\begin{proof}
Let $\{Q_0, \ldots, Q_{N-1}\}$ be the $w$-subsets of $\Z_v$  corresponding to $\code$.

We use double counting:

For some $Q_i \in \code$, let $F = \{((x, x'), \delta) \; : \; x-x' = \delta, x, x' \in Q_i, x \neq x'\}$.

There are $w$ values of $x$ and $w-1$ values of $x'$ and for each pair
of $(x, x')$ there is a unique $\delta = x-x'$.  So $|F| = w(w-1)$. On the other
hand there are $v-1$ possible values of $\delta$, and for each $\delta$ there
are at most $\lambda_a$ pairs of $(x,x')$ such that $x-x' = \delta$.  So
$|F| \le \lambda_a (v-1)$.  This gives the inequality.
\end{proof}

In fact, the proof above says that the auto-correlation value for every $X_i \in \code$ is at least $w(w-1)/(v-1)$. By definition, this auto-correlation value must be at most $\lambda_a$.  Hence we must have

\begin{corollary} \label{cor:labound-eq}
    For an a-proper $(v,w, \lambda_a, \lambda_c)$-OOC $\code$, if  $\lambda_a = w(w-1)/(v-1)$ then every sequence in $\code$ has auto-correlation value exactly $\lambda_a$ for every shift (and hence corresponds to a $(v,w,\lambda_a)$-difference set).
\end{corollary}

\begin{example} \label{eg:labound-eq}
The $(73,9,1,3)$-OOC consisting of the multiplicative cosets of $\langle 2 \rangle $ in $\Z_{73}$, 
$$\langle 2 \rangle = \{1,2,4,6,16,32,37,55,64\},$$
has $\lambda_a =  \frac{w(w-1)}{v-1}$.  All 8 subsets are $(73,9,1)$-difference sets.  
\end{example}

As for an upper bound for $\lambda_a$, any sequence with full cyclic order has $\lambda_a \le w-1$.  
What can we say about an a-proper OOC with $\lambda_a=w-1$?  Suppose $Q \subseteq \Z_v$ corresponds to an OOC sequence with auto-correlation value $\lambda = w-1$.  This means that there is some $\delta \in \Z_v \setminus \{0\}$ such that $\delta$ occurs $w-1$ times as an internal difference in $\Delta(Q)$.  If $\delta$ is coprime to $v$ and therefore has a multiplicative inverse $\delta^{-1}$, then $Q' = \delta^{-1} Q$ has the property that 1 occurs $w-1$ times as an internal difference in $\Delta(Q')$.  Hence $Q'$ is a shift of $\{0,1,\ldots, w-1\}$, and we have

\begin{theorem}
    \label{thm:lambda_a_upperbound}
    Let $\code$ be an a-proper $(v,w,w-1, \lambda_c)$-OOC.  If $\code$ contains a sequence $X=(x_i)$ such that $X$ and $X+\delta = (x_{i+\delta})$ has auto-correlation value $w-1$ for some shift $\delta$ coprime to $v$, then $\code$ is equivalent to an OOC containing the sequence with $w$ consecutive 1s. 
\end{theorem}

If $v$ is prime then all non-zero $\delta$ has a multiplicative inverse.  Hence we have a corollary:

\begin{corollary}
    \label{cor:lambda_a_upperbound}
    Let $\code$ be an a-proper $(v,w,w-1, \lambda_c)$-OOC where $v$ is a prime.  Then $\code$ is equivalent to an OOC containing a sequence with $w$ consecutive 1s.
\end{corollary}

In \cite{YXYang}, $(v,w,w-1, \lambda_c)$-OOC with $\lambda_c=w-1$ are described.  Such an OOC of maximum size is constructed by taking all $w$-subsets of $Z_v$ with full cyclic order, and removing all their cyclic shifts.

\subsection{Bounds on $\lambda_c$}
\label{sub:lambda_c_bounds}

Now we consider the bounds on $\lambda_c$.  Since in the application of OOCs $\lambda_c$ is to be minimised, we consider a lower bound.

\begin{lemma} \label{lem:lcbound}
Let $\code = \{ X_0, \ldots, X_{N-1}\}$ be a $(v,w, \lambda_a,
\lambda_c)$-OOC with $|\code| \ge 2$.   Then
\[ \lambda_{c} \ge \frac{w^2}{v}. \]
\end{lemma}

\begin{proof}
Let $\{Q_0, \ldots, Q_{N-1}\}$ be the $w$-subsets of $\Z_v$ correponding to $\code$.  Again we use double-counting.  For some distinct $Q_i, Q_j$, let $F'= \{((x, y), \delta) \; : \; x-y=\delta, x \in Q_i, y \in Q_j\}$ for some $Q_i$, $Q_j$, $Q_i \neq Q_j$.

There are $w$ values of $x$ and $w$ values of $y$, and for each pair of 
$(x,y)$ there is a unique $\delta=x-y$, so $|F'| = w^2$.  On the other hand there
are $v$ possible values of $\delta$, and there are at most $\lambda_c$ pairs of $(x, y)$
such that $x-y = \delta$. So $|F'| \le v \lambda_c$.  This gives the 
inequality.  
\end{proof}

Similar to the proof of Lemma \ref{lem:labound}, the lower bound is met when every $\delta$ occurs exactly $w^2/v$ times as an external difference in $\Delta(Q_i, Q_j)$.  If $\lambda_c = w^2/v$ for an OOC, then, since all cross-correlation values are at most $\lambda_c$ and at least $w^2/v$, we must have the cross-correlation equal to $w^2/v$ for all pairs of sequences.

\begin{corollary} \label{cor:lcbound-eq}
    For a c-proper $(v,w, \lambda_a, \lambda_c)$-OOC $\code$, if  $\lambda_c = w^2/v$ then every pair of sequences in $\code$ has cross-correlation value exactly $\lambda_c$ for every shift.
\end{corollary}

Recall that OOCs meeting this lower bound with equality were presented in Theorem \ref{thm:lcbound-eq}.

\begin{remark}
We note that \cite[Theorem 1]{ChungKumar} and \cite[Theorem 2]{OmraniMorenoKumar} follow immediately from Lemma \ref{lem:labound} and Lemma \ref{lem:lcbound}.  The first part of both theorems says that $\Phi(v,w,\lambda, \lambda) \le 1$ if $w^2 > v\lambda$.  But Lemma \ref{lem:lcbound} says that the condition $w^2 > v\lambda$ is impossible for cross-correlation value $\lambda$ and thus this rules out the existence of an OOC with more than one sequence. The second part of the Theorem says that $\Phi(v,w, \lambda, \lambda) = 0$ if $w(w-1)>(v-1)\lambda$, and this is Lemma \ref{lem:labound}, which says that the condition is impossible for auto-correlation value $\lambda$.
\end{remark}

We end this section with some extensions of Lemma \ref{lem:disjoint} which describe upper bounds on $\lambda_c$ in terms of the size of the maximum (multiset or set) intersection $\Delta(Q_i) \cap \Delta(Q_j)$.  

\begin{theorem} \label{thm:intersection}
Let $B=\{Q_0, \ldots, Q_{N-1}\}$ be $w$-subsets of $\Z_v$ corresponding to a c-proper $(v, w;\lambda_a, \lambda_c)$-OOC. 
Let $2 \le n \le w$ be an integer.
\begin{enumerate}
\item $\lambda_c=1$ if and only if for all $i \neq j$,  $|\Delta(Q_i) \cap \Delta(Q_j)|=0$.
\item If the multiset intersection $|\Delta(Q_i) \cap \Delta(Q_j)|< n(n-1)$ for all $i \neq j$ then $\lambda_c<n$.
\item If the set intersection $|\Delta(Q_i) \cap \Delta(Q_j)|< n-1$ for all $i \neq j$ then $\lambda_c<n$.
\item Let $v=p$ where $p$ is prime and let $n < \frac{p+1}{2}$. 
If the set intersection $|\Delta(Q_i) \cap \Delta(Q_j)| < 2n-1$ for all $i \neq j$ then $\lambda_c <n$.
\end{enumerate}
\end{theorem}

\begin{proof} 
\begin{enumerate}
    \item This is Lemma \ref{lem:disjoint}.
    \item Suppose $\lambda_c \ge n$, and $\gamma \in \Z_v\setminus\{0\}$ occurs $n$ times as an external difference in $\Delta(Q_i,Q_j)$ for some $i \neq j$,  say $\gamma=x_1-y_1= \cdots = x_n-y_n$ where $D_i=\{x_1, \ldots, x_n \} \subseteq Q_i$ and $D_j=\{y_1, \ldots, y_n \} \subseteq Q_j$.  Clearly the $n$ elements of $D_i$ are all distinct, and similarly for $D_j$ (we have $n \leq w$).   Since for each pair $(k,l)$ with $1 \leq k,l \leq n$, we have $x_k-x_l=y_k-y_l$, so $\Delta(D_i)=\Delta(D_j) = \Delta$, say, and $\Delta \subseteq \Delta(Q_i) \cap \Delta(Q_j)$.

    Since there are $n(n-1)$ differences in the multiset $\Delta$, we have that the size of the multiset $\Delta(Q_i) \cap \Delta(Q_j)$ is at least $n(n-1)$.

    \item The first part of the proof is the same as the previous part.  We consider only the size of the set intersection $\Delta(Q_i) \cap \Delta(Q_j)$.  Since $|D_i| = n$, fixing any $x_k \in D_i$ and counting $\{x_k - x_j \; | \; x_j \in D_i \setminus{\{x_k\}}\}$, we have $|\Delta(Q_i) \cap \Delta(Q_j)| \ge |\Delta| \ge n-1$.

    We note that $\Delta=n-1$ can be attained, for example, if $D_i=\{0,4,8\}$ in $\Z_{12}$ then $\Delta=\{4,8\}$.

    \item Again, here we consider the size of the set intersection $\Delta(Q_i) \cap \Delta(Q_j)$.  Since $p$ is prime, $|D_i| = n = |-D_i|$, and the Cauchy-Davenport Theorem states that $\Delta(D_i) = |D_i| + |-D_i| \ge 2n-1$.  
\end{enumerate}
\end{proof}

In \cite{MomBuratti, BurMomPas, BurPasWu}, upper bounds are obtained for $(v,w, \lambda_a, 1)$-OOCs, by making use of part 1 of the theorem:  since the intersections of the $\Delta(Q_i)$ are empty, knowing the cardinality of the sets $\Delta(Q_i)$ yields an upper bound of the number of $Q_i$ in $\{1, \ldots, v-1\}$.  It would be of interest to extend this to $\lambda_c>1$ via analysis of the intersection of the $\Delta(Q_i)$.  The converse appears to be only obtainable for $\lambda_c=1$.

\section{Conclusion and Further Work}
\label{sec:conclusion}

Considering OOCs with $\lambda_a$ and $\lambda_c$ ``decoupled" is an interesting combinatorial question, which at present has not received as much attention as the case when $\lambda_a=\lambda_c$.  Moreover, since the motivating communications problem does not require $\lambda_a=\lambda_c$, such OOCs are likely to have useful applications. It is of interest to obtain further constructions of these.

In Section \ref{sec:boundsandconstructions} we surveyed bounds and constructions in the literature, and considered why these bounds may not be very accurate.  Considering separate bounds on the auto- and cross-correlation values in Section \ref{sec:bounds} yielded structural insights (and led to the construction in Theorem \ref{thm:lcbound-eq}).  It would be of interest to see further development of separate bounds, and more constructions which meet these new bounds.  

Many constructions and bounds for OOCs with $\lambda_c=1$ in the literature rely heavily on Lemma \ref{lem:disjoint}. It would be useful to obtain a suitable generalisation of this for $\lambda_c>1$, which has the potential to yield better bounds on the sizes of OOCs. This is likely to require a structural analysis of the internal and external differences, an initial exploration of which is given in Theorem \ref{thm:intersection}.   

We have seen (Section \ref{sub:OOCsubsets}) that OOCs may be expressed in terms of subsets of $\mathbb{Z}_v$ satisfying certain conditions.  This definition can be generalized to subsets of general groups in the natural way.  Indeed, difference families and their variants are defined over general groups and there are many known constructions (including in non-abelian groups). Such a generalisation was presented by Buratti in \cite{Bur}: a $(v,w,\lambda, \lambda)$-OOC over an additive group $G$ was defined in terms of maps from $G$ to $\mathbb{Z}_2$, and an expression in terms of subsets of $G$ was given for these OOCs.

Many definitions in this paper can be generalised from $\mathbb{Z}_v$ to arbitrary groups, for example, the concept of SDF (Definition \ref{SDFdef}). Moreover, it is particularly natural to extend cyclotomic OOC results: many  constructions in $\mathbb{Z}_p$ which are defined in terms of cyclotomic classes of $GF(p)$ where $p$ is prime will be generalisable to the additive group of $GF(q)$ where $q$ is a prime power.  Example \ref{ex:SEDF} is a natural example.  More examples are given in \cite{Bur}.  It would be of combinatorial interest to see further constructions in arbitrary groups: for example in the non-abelian setting.

We considered the question of isomorphism and equivalence of OOCs and their developments.  It is an open question whether there are isomorphic non-equivalent OOCs defined over general groups.  In a similar vein, we discussed the relationship between OOCs and packings in Section \ref{sub:packing} and Theorem \ref{thm:packing} gives an equivalence of non-degenerate $(v, w, t-1, t-1)$-$\OOC$s and cyclic $t$-$(v,w,1)$-packings.  Is it possible to make similar statements about OOCs defined over general groups?

\end{document}